\documentclass[12pt, reqno]{amsart}
\usepackage{amsmath, amsthm, amscd, amsfonts, amssymb, graphicx, color, mathrsfs}

\usepackage[bookmarksnumbered, colorlinks, plainpages]{hyperref}
\usepackage[all]{xy}
\usepackage{slashed}
\usepackage{mathabx}
\usepackage{tipa}
\usepackage{soul}
\usepackage{cancel}
\usepackage{ulem}

\newtheorem{theorem}{Theorem}[section]
\newtheorem{lemma}[theorem]{Lemma}

\theoremstyle{definition}
\newtheorem{definition}[theorem]{Definition}

\theoremstyle{remark}
\newtheorem{remark}[theorem]{Remark}
\numberwithin{equation}{section}

\begin{document}
\setcounter{page}{1}

\title[Schr\"odinger equation on the sphere ]{Sharp Strichartz estimates for the Schr\"odinger equation on the sphere }

\author[D. Cardona]{Duv\'an Cardona}
\address{
  Duv\'an Cardona S\'anchez:
  \endgraf
  Department of Mathematics: Analysis, Logic and Discrete Mathematics
  \endgraf
  Ghent University, Belgium
  \endgraf
  {\it E-mail address} {\rm duvanc306@gmail.com, duvan.cardonasanchez@ugent.be}
  }
  
  \author[L. Esquivel]{Liliana Esquivel}
\address{
 Liliana Esquivel:
  \endgraf
 Gran Sasso Science Institute, CP 67100
  \endgraf 
  L' Aquila, Italia.
  \endgraf
  {\it E-mail address} {\rm liliana.esquivel@gssi.it}
  }

\subjclass[2010]{Primary: 35Q40, { Secondary: 42B35, 42C10, 35K15}.}

\keywords{}

\thanks{The first author was supported by the FWO Odysseus 1 grant G.0H94.18N: Analysis and Partial Differential Equations of Professor Michael Ruzhansky.}

\begin{abstract} 
In this contribution we investigate the  Schr\"ordinger equation associated to the Laplacian on the sphere in the form of sharp Strichartz estimates. We will provided   simple proofs for our main theorems  using purely  the  $L^2\rightarrow L^p$ spectral estimates for the operator norm of the spectral projections (associated to the spherical harmonics)  proved in \cite{KS}. A sharp index of regularity is established for the initial data in spheres of arbitrary dimension $d\geq 2$.  

\end{abstract} \maketitle

\section{Introduction}   

In this work we provide
sharp Strichartz estimates for the Schr\"odinger equation for the Laplacian on the sphere $\mathbb{S}^d,$ 
$\Delta_{\mathbb{S}^d},$ which is given by
\begin{equation}\label{PVI}
   \textnormal{(PVI):}
     \begin{cases}
      iu_{t}+\Delta_{\mathbb{S}^d}u=0,&\quad (t,x)\in \mathbb{T}\times \mathbb{S}^d,
      \\
       u(0,x)=f(x), &\quad x\in \mathbb{S}^d,
     \end{cases}
\end{equation} where $\mathbb{T}\cong [0,2\pi),$ $0\sim 2\pi,$ denotes the one dimensional torus. The fractal nature of the solutions for \eqref{PVI} has been discussed, e.g. in Taylor \cite{Taylor2003} and references therein. To illustrate our results let us recall, that \eqref{PVI} is an analogy of the standard Schr\"odinger equation on $\mathbb{R}^n,$
\begin{equation}\label{SEq}
iu_t(t,x)+\Delta_{\mathbb{R}^n} u(t,x)=0, 
\end{equation} with initial data $u(0,\cdot)=f.$ The solution for \eqref{SEq}, satifies the following Strichartz estimate
\begin{equation}\label{KT'}
\Vert e^{it\Delta_{\mathbb{R}^n}}f\Vert_{L^q[(0,\infty), L_x^p(\mathbb{R}^d)]}\leq C\Vert f\Vert_{L^2(\mathbb{R}^n)},
\end{equation}
which holds, if and only if, $1/q=n/2(1/2-1/p),$ for $n=1$ and $2\leq p\leq \infty,$ $n=2$ and $2\leq p<\infty,$ and, $2\leq p<\frac{2n}{n-2}$ if $n\geq 3.$
This fundamental result proved by Keel and Tao \cite{KeelTao} is the source of inspiration for several extensions of \eqref{SEq} to general compact and non-compact manifolds.

The key point for obtaining \eqref{KT'}, is the following estimate
\begin{equation}\label{Analogue}
    \Vert e^{it\Delta_{\mathbb{R}^n}}\Vert_{L^1(\mathbb{R}^n)\rightarrow L^\infty(\mathbb{R}^n) }\lesssim_{n} C|t|^{-\frac{n}{2}},
\end{equation}which expresses the dispersive property of the linear Schr\"odinger equation on $\mathbb{R}^n.$ If $\mathbb{R}^n$ is replaced by a compact manifold $M,$ the analogue of the dispersion estimate \eqref{Analogue} fails globally in time, as the example of constant solutions shows (see Burq, Gerard and Tzvetkov \cite{BurqGerardT} for details). It is less obvious, but still true, that it also fails locally in time (see the discussion in the beginning of Section 3 of \cite{BurqGerardT}). For the case  of the $d$-dimensional torus, $M=\mathbb{T}^d,$ Bourgain \cite{Bourgain93,Bourgain94,Bourgain99} proposed a different approach to Strichartz estimates by using the Fourier series representation of the Schr\"odinger group,
\begin{equation*}
   u(t,x)\equiv e^{it\Delta_{\mathbb{T}^d}}f(x):=\sum_{k\in \mathbb{Z}^d} e^{i2\pi (t|k|^2+k\cdot x)}\widehat{f}(k). 
\end{equation*} Notice that this representation shows that the solution is also periodic in time. By computing explicitly 
\begin{equation}\label{Bourgain}
     \Vert u \Vert^4_{L^4(\mathbb{T}\times \mathbb{T}^d )}= \Vert u^2 \Vert^2_{L^2(\mathbb{T}\times \mathbb{T}^d )},
\end{equation}   Bourgain obtains similar Strichartz estimates as in the fundamental paper due to Keel and Tao \cite{KeelTao}.

In general, on a compact manifold without boundary $M,$ it was proved by Burq, Gerard and Tzvetkov \cite{BurqGerardT} the following estimate
\begin{align}\label{Burq}
\Vert e^{it\Delta_{M}}f\Vert_{L^q[[0,T], L_x^p(M)]}\lesssim_{T,p,q,s} \Vert f\Vert_{W^{s}(M)},\,\,s\geq \frac{1}{q},
\end{align}for $1/q=\textnormal{dim}(M)/2(1/2-1/p),$ $q\geq 2,$ and  $p<\infty.$ As in our case, Taylor in \cite{Taylor2003} considered the case of spheres $M=\mathbb{S}^d,$ where the author obtained sharp $L^p$-estimates on the solution operator at those times which are rational multiples of $\pi.$ This idea that the profile of linear dispersive equations depend on the algebraic properties of time have been further exploited in the papers of Kapitanski-Rodniaski \cite{KapitanskiRodniaski99}, Taylor  \cite{Taylor2003} and Rodnianski \cite{Rodnianski} where the fractal nature for the solution of the Schr\"odinger equation on spheres and torus were considered. 

Our main goal is to study \eqref{PVI}, by using the $L^2\rightarrow L^p$ bounds for the spectral resolution of  the Laplacian $\Delta_{\mathbb{S}^d}.$ More precisely, we deal with the problem of establishing an analogy of \eqref{SEq} in the case of the Laplacian on the sphere $\mathbb{S}^d.$  Our main results are, the following Theorem \ref{duvantheorem}, and Theorem \ref{LiliTheorem}, where the Schr\"odinger equation with potential \eqref{PVI'}, (this means, associated to $-\Delta_{\mathbb{S}^d}+V(x,t),$ for a suitable potential $V(x,t)$) is considered. 

\begin{theorem}\label{duvantheorem}
Let $d\geq 2$ and $p,q$ satisfy $2\leq p\leq \infty,$ and $2\leq q<\infty.$ Then the following Strichartz estimate
\begin{equation}\label{TheoremM}
     \Vert e^{it\Delta_{\mathbb{S}^d}}f(z) \Vert_{L^{p}_z[\mathbb{S}^{d},\,L^q_t(\mathbb{T})]}\lesssim_{p,q,s} \Vert f \Vert_{W^{s}(\mathbb{S}^d)},\footnote{Observe that under the symmetric compromise  $p=q,$ $\Vert u(t,z) \Vert_{L^{p}_z[\mathbb{S}^{d},\,L^q_t(\mathbb{T})]}=\Vert u(t,z) \Vert_{L^{p}(\mathbb{S}^{d}\times \mathbb{T})},$ and for $p=4,$ $  \Vert u \Vert^4_{ L^4(\mathbb{S}^d\times\mathbb{T}  )}= \Vert u^2 \Vert^2_{L^2(\mathbb{S}^d\times \mathbb{T})} .$ We use the notation $A\lesssim B$ to indicate $A\leq cB,$ where $c>0$ does not depend on fundamental quantities.}
\end{equation}holds true for $s\geq \varkappa_{p,q},$ where
$$   \varkappa_{p,q}: = \left\{
     \begin{array}{lr}
       \frac{d-1}{2}\left(\frac{1}{2}-\frac{1}{p}\right)+\left(\frac{1}{2}-\frac{1}{q}\right), &   2\leq p\leq \frac{2(d+1)}{d-1},\\
       d\left(\frac{1}{2}-\frac{1}{p}\right)-\frac{1}{q}, &   p>\frac{2(d+1)}{d-1}. \\

     \end{array}
   \right.  $$ 
 For $q=2,$ and $s<\varkappa_{p,2},$ the estimate in \eqref{TheoremM} does not hold for any non-trivial initial data $f\in L^2(\mathbb{S}^d)$. So, in this case, the regularity order $\varkappa_{p,2} $ is sharp in any dimension $d,$ in the sense that \eqref{TheoremM}  does not hold for all $s<\varkappa_{p,2} $.
\end{theorem}
By keeping the notation $\varkappa_{p,2}$ in Theorem \ref{duvantheorem}, we consider the case of Schr\"odinger equation associated with the Hamiltonian $$\mathbb{H}:=-\Delta_{\mathbb{S}^d}+V(x,t).$$ Indeed, we prove:
\begin{theorem}\label{LiliTheorem} Let us consider the following Schr\"odinger equation with potential
\begin{equation}\label{Lilicase}
   \textnormal{(PVI'):}
     \begin{cases}
      iu_{t}=(-\Delta_{\mathbb{S}^d}+V(x,t))u,&\quad (t,x)\in \mathbb{T}\times \mathbb{S}^d,
      \\
       u(0,x)=f(x), &\quad x\in \mathbb{S}^d
     \end{cases}
\end{equation} where    $\Vert V \Vert_{L^{q}_x(\mathbb{S}^d,L^\infty_t(\mathbb{T}))}$ is small enough. Let us assume that  $\frac{1}{q}+\frac{1}{p/2}=1.$ Then, there exists a unique solution  $u\in {C}(\mathbb{T}, W_x^{s}(\mathbb{S}^d))\cap L^p_x(\mathbb{S}^d, L^2_t(\mathbb{T}))$ for \eqref{Lilicase}, provided that $s\geq \varkappa_{p,2}$
\end{theorem}

This paper is organised as follows. In Section \ref{Preliminaries} we present some basic properties about the spectral decomposition of the Laplacian and we introduce some auxiliar (Triebel-Lizorkin) spaces, that will be useful in our analysis. In Section \ref{t1}  we prove our Theorem \ref{duvantheorem}. Finally, in Section  \ref{potential} we consider the Schr\"odinger equation with potential. 
We finish this introduction by  discussing briefly our main results.
\begin{remark} For $M=\mathbb{S}^d,$ let us observe that \eqref{Burq} is in general a better estimate (in the case when $p>2(d+1)/(d-1)$) that the obtained in Theorem
\ref{duvantheorem} in the case, where, $$1/q=d/2(1/2-1/p),\footnote{By following the usual nomenclature,   such pair $(p,q)$ is called admissible.}$$ but, both estimates agree for $2\leq p\leq 2(d+1)/(d-1).$ However, the main point in Theorem \ref{duvantheorem} is the sharpness of the index $\varkappa_{p,2}$ in any dimension, and that the estimate that we obtain is in general independent of the known restriction $1/q=\textnormal{dim}(M)/2(1/2-1/p),$ between the dimension $d,$ and the pair $(p,q).$ Several sharp results on the sphere are know only for slow dimensions $d,$ (see e.g. Taylor \cite{Taylor2003}), or for $p=q=4$ (see Burq, Gerard and Tzvetkov \cite{BurqGerardT}), but the sharp Sobolev regularity order  $\varkappa_{p,2}$ is to our knowledge, a  new result in this context and also in the context of Theorem \ref{LiliTheorem} for the Hamiltonian $\mathbb{H}$. 
\end{remark}
\begin{remark}We will discuss some implications of our results to general manifolds by following a similar analysis as in Ruzhansky and Turunen \cite[page 578]{RuzhanskyTurunenBook2010}.
Under the existence of a $C^\infty$-diffeomorphism $i:\mathbb{S}^d\rightarrow M,$ where $M$ is a Riemannian manifold, (these kind of diffeomorphisms are known for $d=3,$ with $M$ being a connected a simply connected manifold in view of the Perelman proof \cite{Perelman1,Perelman2,Perelman3,Perelman4} for the Poincar\'e conjecture \cite{Poincare} and in higher dimensions), if $X_{1},\cdots,X_{d}$ is an orthonormal basis for the tangent bundle $T\mathbb{S}^d,$ in such a way that
\begin{equation}
    \Delta_{\mathbb{S}^d}:=-\sum_{1\leq j\leq d}X_j^2,
\end{equation}let us denote by
\begin{equation*}
    \Delta_{M}:=-\sum_{1\leq j\leq d}Y_j^2,
\end{equation*}  the second order differential operator induced by $i.$ This means that, $Y_{j}:=i_{*}(X_{j}),$ is the vector field induced by the differential of $i,$ which is a diffeomorphism between tangent bundles, $i_{*}:T\mathbb{S}^d\rightarrow TM.$ Theorem \ref{duvantheorem} implies the estimate
\begin{equation}\label{TheoremMM}
     \Vert e^{it\Delta_M}u_0 \Vert_{L^{p}_z[M,\,L^q_t(\mathbb{T})]}\leq C\Vert u_0 \Vert_{W^{s}(M)},
\end{equation} for $s\geq \varkappa_{p,q},$ with $\varkappa_{p,2}$ being sharp, and also Theorem \ref{LiliTheorem}, implies the existence of a unique $u\in {C}(\mathbb{T}, W_x^{s}(M))\cap L^p_x(M, L^2_t(\mathbb{T}))$ satisfying the differential problem 
\begin{equation}\label{Lilicase2}
   \textnormal{(PVI''):}
     \begin{cases}
      iu_{t}=(-\Delta_{M}+V(x,t))u=0,&\quad (t,x)\in \mathbb{T}\times M,
      \\
       u(0,x)=u_0(x), &\quad x\in M,
     \end{cases}
\end{equation} for a potential $V$ with  $\Vert V \Vert_{L^{q}_x(M,L^\infty_t(\mathbb{T}))}$ small enough.

\end{remark}

\section{preliminaries}\label{Preliminaries} 
Now we present some basics about the spectral decomposition of the Laplacian on the sphere, for this we follow Schoen and Yau \cite{T}.
Let $d\geq 1,$  $d\in \mathbb{N}.$ Let $$\mathbb{S}^{d}:=\{x=(x_0,x_1,\cdots,x_d)\in \mathbb{R}^{d+1}:|x_0|^2+|x_1|^2+\cdots+|x_d|^2=1\},$$ be the $d$-dimensional unit sphere contained in $\mathbb{R}^{d+1}$ and $\mathcal{H}^{d}_n$ be the space of spherical harmonics of degree $n.$ There is a standard metric $g$ on $\mathbb{S}^{d},$ induced by the standard metric on $\mathbb{R}^{d+1},$ and conrrespondly, a Laplace operator $\Delta_\mathbb{S}^d$ which can be defined, initially on $C^\infty(\mathbb{S}^{d})$ as follows. For $f\in C^\infty(\mathbb{S}^d),$ let $\tilde{f}(x):=f(x/|x|),$ then
\begin{equation*}
 \Delta_{\mathbb{S}^d} f:=\Delta_{\mathbb{R}^{d+1}}\tilde{f}|_{\mathbb{S}^d},   
\end{equation*}where, by following the usual nomenclature, $$\Delta_{\mathbb{R}^{d+1}}:=-\partial_{x_0}^2-\partial_{x_1}^2-\cdots- \partial_{x_d}^2$$ is the standard positive Laplacian on $\mathbb{R}^{d+1}.$ It can be shown that for $r=|x|,$ one can write
\begin{equation}
    \Delta_{\mathbb{R}^{d+1}}\tilde{f}(x)=r^{-2}(\Delta_{\mathbb{S}^d} f).
\end{equation}In polar coordinates $r=|x|,$ $\omega=x/r,$
\begin{equation*}
   \Delta_{\mathbb{S}^d}= r^2\Delta_{\mathbb{R}^{d+1}}-
    r^2\partial_{r}^2-dr\partial_{r}.
    \end{equation*}
It is well known, from the Fourier analysis associated to $\Delta_{\mathbb{S}^d},$ the spectral decomposition
\begin{equation*}
    L^{2}(\mathbb{S}^{d})=\bigoplus_{n=0}^\infty \mathcal{H}^{d}_n,\,\,\dim(\mathcal{H}^{d}_n)\sim n^{d-1},
\end{equation*} where $L^2(\mathbb{S}^d)\equiv L^2(\mathbb{S}^d,dx)$ is the usual $L^2$-space\footnote{As usually,  defined by those measurable functions on $\mathbb{S}^d,$ such that  $\Vert f\Vert_{L^2(\mathbb{S}^d)}:=\left(\,\int\limits_{\mathbb{S}^d}|f(\omega)|^2d\omega\right)^{\frac{1}{2}}<\infty.$ In general, on a measure space $(X,dx),$ we use $\Vert f\Vert_{L^p(X)}:=\left(\,\int\limits_{X}|f(x)|^pdx\right)^{\frac{1}{p}},$ for all $1\leq p<\infty.$} associated to the surface measure  on the sphere, $dx:=d\sigma_{\mathbb{S}^d}(x),$ and that, $\mathcal{H}^{d}_n$ is  the eigenspace of the Laplace-Beltrami operator $\Delta_{\mathbb{S}^d}$ associated  with the eigenvalue $\lambda_{n}=n(n+d-1).$ Let  $H^d_{n}$ be the spectral projection to the subspace $\mathcal{H}^{d}_n$ for every $n\in \mathbb{N}_{0}.$ Then, for every $f\in L^2(\mathbb{S}^{d})$ we can write 
\begin{equation*}
    f=\sum_{n=0}^{\infty}H^d_{n}f,
\end{equation*} where the series converges in the $L^2(\mathbb{S}^d)$-norm. In particular, the spectral theorem for $\Delta_{\mathbb{S}^d}$ implies 
\begin{equation*}
 m(\Delta_{\mathbb{S}^d})=   \sum_{n=0}^{\infty}m(\lambda_n)H^d_{n}f,
\end{equation*}for every continuous function $m$ on $\mathbb{R},$ where the domain of $m(\Delta_{\mathbb{S}^d})$ is given by
\begin{equation*}
    \textnormal{Dom}(m(\Delta_{\mathbb{S}^d})):=\{f\in  L^2(\mathbb{S}^{d}):\sum_{n=0}^\infty|m(\lambda_n)|^2\Vert H^d_{n}f\Vert^2_{L^2(\mathbb{S}^d)}<\infty \}.
\end{equation*}
By observing that, for  $m(\lambda):=e^{\pm it\lambda },$ $\textnormal{Dom}(m(\Delta_{\mathbb{S}^d}))=L^{2}(\mathbb{S}^d),$  the Schr\"odinger propagator $e^{it \Delta_{\mathbb{S}^d} }:L^2(\mathbb{S}^{d})\rightarrow L^2(\mathbb{S}^d)$ is given by
\begin{equation*}
   e^{it \Delta_{\mathbb{S}^d}}f:= \sum_{n=0}^{\infty}e^{i\lambda_n t}H^d_{n}f,
\end{equation*} and for $f\in C^\infty(\mathbb{S}^d)$ the solution to \eqref{PVI} is given by
\begin{equation*}\label{equsol}
    u(t,x)=\sum_{n=0}^{\infty}e^{i\lambda_n t}H^d_{n}f(x),\quad (t,x)\in \mathbb{T}\times \mathbb{S}^d.
\end{equation*}

In this paper we want to estimate the mixed norms $L^p_x(L^q_t)$ of  solutions to Schr\"odinger equations by using the following version of Triebel-Lizorkin associated to the family of spectral projections of the Laplacian $\Delta_{\mathbb{S}^d}$.
 \begin{definition}
Let us consider $0<p\leq \infty,$ $r\in\mathbb{R}$ and $0<q\leq \infty.$ The Triebel-Lizorkin space associated to the family of projections $H_{n}^d,$ $n\in  \mathbb{N}_0,$ and to the parameters $p,q$ and $r$ is defined by those complex functions $f$ satisfying
\begin{equation}\label{5}
\Vert f\Vert_{{F}^r_{p,q}(\mathbb{R}^n)}\equiv \Vert f\Vert_{{F}^r_{p,q}(\mathbb{R}^n,\,\,(H_{n}^d)_{n\in \mathbb{N}})}:=\left\Vert \left(\sum_{n=0}^\infty n^{rq}|H_n^d f|^{q}\right)^{\frac{1}{q}}\right\Vert_{L^p(\mathbb{S}^d)}<\infty.
\end{equation}
 \end{definition}
 The definition considered above differs from those arising with dyadic decompositions on compact homogeneous manifolds (see e.g. Nursultanov, Ruzhansky and Tikhonov \cite{NRT}). Indeed, when working with dyadic decompositions one recover the Littlewood-Paley theorem, about the equivalence of certain dyadic Triebel-Lizorkin spaces with the Lebesgue spaces $L^p(\mathbb{S}^d),$ which with the current definition we do not necessarily have. The following are natural embedding properties of such spaces.   Sobolev spaces $W^{s,p}$  in $L^p(\mathbb{S}^d)$-spaces and associated to $\Delta_{\mathbb{S}^d},$ can be defined by the norm 
 \begin{equation}\label{Sobp}
 \Vert f\Vert_{W^{s,p}}\equiv \Vert \Delta^s_{\mathbb{S}^d} f\Vert_{L^p}:=\left(\,\int\limits_{\mathbb{S}^d}|\Delta^s_{\mathbb{S}^d} f(x)|^pdx \right)^{\frac{1}{p}}\footnote{ Here, $W^{s,p}:=\textnormal{Dom}(\Delta^s_{\mathbb{S}^d})$ is considered as the completion of $C^\infty(\mathbb{S}^d)$ with respect to the norm defined in  \eqref{Sobp}.}.    
 \end{equation}
  We will use the notation $W^s:={W^{s,2}}$ for the Sobolev space modelled on $L^2(\mathbb{S}^d).$  Indeed, we  have
 \begin{equation*}
     \Vert u\Vert_{W^s(\mathbb{S}^d)}\asymp\left( \sum_{\ell=0}^\infty \ell^s\Vert H_\ell^d u\Vert_{L^2(\mathbb{S}^d)}  \right)^\frac{1}{2}\footnote{Where we use $\asymp$ because $\ell\sim\lambda_\ell:= [\ell(\ell+d-1)]^{\frac{1}{2}}$ for $\ell\rightarrow\infty.$} .
 \end{equation*}

 The following properties can be deduced  by following a similar analysis as in  \cite{NRT}.
 \begin{itemize}
\item[-] ${{F}}^{r+\varepsilon}_{p,q_1}\hookrightarrow {{F}}^{r}_{p,q_1}\hookrightarrow {{F}}^{r}_{p,q_2}\hookrightarrow {F}^{r}_{p,\infty},$  $\varepsilon>0,$ $0<p\leq \infty,$ $0<q_{1}\leq q_2\leq \infty.$
\item[-]  ${F}^{r+\varepsilon}_{p,q_1}\hookrightarrow {F}^{r}_{p,q_2}$, $\varepsilon>0,$ $0<p\leq \infty,$ $1\leq q_2<q_1<\infty.$
\item[-] ${F}^0_{2,2}=L^2$ and consequently, for every $s\in\mathbb{R},$ $W^{s}={F}^s_{2,2}.$  
\end{itemize}
It is a very special fact, how at rational multiples of $\pi,$ the propagator $e^{it \Delta_{\mathbb{S}^d}}$ behaves on Sobolev spaces. Indeed, let us remark that (see Proposition 2.1 of Taylor \cite{Taylor2003}) for all $1<p<\infty,$ and all $s\in \mathbb{R},$
\begin{equation}
    e^{2\pi i(\frac{m}{2k})\Delta_{\mathbb{S}^d}}:W^{s,p}(\mathbb{S}^d)\rightarrow W^{s-(d-1)|\frac{1}{2}-\frac{1}{p}|,p}(\mathbb{S}^d),
\end{equation}extends to a bounded operator. For the sharpness of this estimate we refer the reader to  Taylor \cite[page 148]{Taylor2003}.

\section{Proof of  Theorem \ref{duvantheorem}}\label{t1}
In this section we will prove one of our main results.   We will see that information about the Strichtaz estimates for the Schr\"odinger propagator, is encoded in the $L^2\rightarrow L^p$-operator norm of the projections ${H}_n^d,$ and this will essential in order to give a short proof for  Theorem \ref{duvantheorem}.

The space $L^2_{\textnormal{fin}}(\mathbb{S}^{d})$ consists of those finite linear combinations of elements in the sub-spaces  $\mathcal{H}_n^d.$ 
\begin{lemma}
Let us consider $f\in {F}^0_{p,2} (\mathbb{S}^{d} )$, then for all $1\leq p\leq \infty$  we have 
\begin{equation}\label{eql2lp}
\Vert u(t,z) \Vert_{L^p_z[\mathbb{S}^{d},\,L^2_t(\mathbb{T})]}=\sqrt{2\pi}\Vert f \Vert_{{F}^0_{p,2}(\mathbb{S}^{d})}.
\end{equation}
\end{lemma}
\begin{proof}
Let us consider $f\in {F}^0_{p,2}(\mathbb{S}^{d}).$ By a standard argument of density we only need to assume $f\in L^2_{\textnormal{fin}}(\mathbb{S}^{d}).$ The solution $u(t,z)$ for \eqref{PVI} is given by
\begin{equation}
u(t,z)=\sum_{n=0}^{\infty}e^{it \lambda_n}H_{n}^df(z).
\end{equation}
Let us note that the previous sums runs over a finite number of $n's$ because $f\in L^2_{\textnormal{fin}}(\mathbb{S}^{d}).$ So, we have,
\begin{align*}
\Vert u(t,z) \Vert_{L^2_t(\mathbb{T})}^2&=\int\limits_{0}^{2\pi}u(t,z)\overline{u(t,z)}dt=\int\limits_{0}^{2\pi} \sum_{n,n'}e^{it(\lambda_n-\lambda_{n'})}H_{n}^df(z)\overline{ H_{n'}^df(z)  }  dt .
\end{align*}Observe that if $n\neq n',$ then $$\lambda_n=n(n+d-1)\neq n'(n'+d-1)=\lambda_{n'},$$ and consequently, from the  $L^2$-orthogonality of the trigonometric polynomials we get,
\begin{align*}
\Vert u(t,z) \Vert_{L^2_t(\mathbb{T})}^2&=\int\limits_{0}^{2\pi}u(t,z)\overline{u(t,z)}dt
= \sum_{n,n'}\int\limits_{0}^{2\pi}e^{it(\lambda_n-\lambda_{n'})}H_{n}^df(z)\overline{ H_{n'}^df(z)  }  dt\\
&=\sum_{n=0}^{\infty}2\pi\,\cdot|H_{n}^df(z)|^2,
\end{align*} where we have changed summation with integration because $f\in L^2_{\textnormal{fin}}(\mathbb{S}^{d}).$
Thus, we conclude the following fact
\begin{equation}\label{identity1}
\Vert u(t,z) \Vert_{L^2_t(\mathbb{T})}=\left(\sum_{n=0}^\infty2\pi\,\cdot|H_{n}^df(z)|^2\right)^{\frac{1}{2}},\,\,\,  f\in L^2_{\textnormal{fin}}(\mathbb{S}^{d}).
\end{equation}
Finally,
\begin{equation}
\Vert u(t,z)\Vert_{L_2(\mathbb{S}^{d},L^2_t(\mathbb{T}) )}=\sqrt{2\pi}\Vert f\Vert_{{F}^0_{p,2}(\mathbb{S}^{d})},
\end{equation}as claimed.

\end{proof}
The following lemma, allows us to compare the $L^p_{z}(L^q_{t})$-mixed norms of the solution $u(t,z)$ and some Triebel-Lizorkin norms of the initial data.

\begin{lemma}\label{T1}
Let $p,q$ and $s_q$ be such that $0<p\leq \infty,$ $2\leq q<\infty$ and $s_{q}:=\frac{1}{2}-\frac{1}{q}.$ Then
\begin{equation}
C_p'\Vert f\Vert_{{F}^0_{p,2}}\leq \Vert u(t,z)\Vert_{L^p_{z}(\mathbb{S}^{d}, L^q_{t}(\mathbb{T}))} \leq C_{p,s}\Vert f\Vert_{{F}^s_{p,2}}, 
\end{equation}
holds true for every $s\geq s_{q}.$ 
\end{lemma}
\begin{proof}
Let us  consider, by a density argument, the initial data $f\in L^2_f(\mathbb{R}^n).$ In order to estimate the norm  $\Vert u(t,z) \Vert_{L^p_z[\mathbb{S}^{d},\,L^q_t(\mathbb{T})]}$ we can use the Wainger Sobolev embedding Theorem:
\begin{equation}
\left\Vert \sum_{n\in \mathbb{Z},n\neq 0}|n|^{-\alpha}\widehat{F}(n)e^{in t}\right\Vert_{L^q(\mathbb{T})}\leq C\Vert F\Vert_{L^r(\mathbb{T})},\,\,\alpha:=\frac{1}{r}-\frac{1}{q}.
\end{equation}
For $s_q:=\frac{1}{2}-\frac{1}{q}$  we have
\begin{align*}
\Vert u(t,z)\Vert_{L^q(\mathbb{T})} &=\left\Vert \sum_{n=0}^{\infty} e^{it\lambda_n}H_n^df(z)  \right\Vert_{L^q(\mathbb{T})} \\
&\leq C\left\Vert \sum_{n=0}^{\infty} n^{s_q}e^{it\lambda_n}H_n^df(z)  \right\Vert_{L^2(\mathbb{T})}\\
&=C\left\Vert \sum_{n=0}^{\infty} e^{it\lambda_n}H_n^d[\mathcal{L}^{s_q}f(z)]  \right\Vert_{L^2(\mathbb{T})}\\
&=C\left( \sum_{n=0}^{\infty} |H_n^d[\mathcal{L}^{s_q}f(z)|^2  \right)^{\frac{1}{2}}\\
&:=T'(\mathcal{L}^{s_q}f)(z),
\end{align*}where we have denoted
\begin{equation*}
 \mathcal{L}^{s_q}:=   \sum_{n=0}^{\infty}n^{s_q}H^d_{n}.
\end{equation*}
 So, we have
\begin{align*}
\Vert u(t,z) \Vert_{L^p_z[\mathbb{S}^{d},\,L^q_t(\mathbb{T})]}&\leq C\Vert T'(H^{s_q}f) \Vert_{L^p(\mathbb{S}^{d})}\\&= C\Vert \mathcal{L}^{s_q}f\Vert_{{F}^0_{p,2}(\mathbb{S}^{d})}\\
&=C\Vert f\Vert_{{F}^{s_q}_{p,2}(\mathbb{S}^{d})}.
\end{align*}
We end the proof by taking into account the embedding ${F}^s_{p,2}\hookrightarrow {F}^{s_q}_{p,2}$ for every $s\geq s_q,$ and the following estimate for all $2\leq q<\infty,$
\begin{align*}
\Vert f\Vert_{{F}^0_{p,2}}&\lesssim  \Vert T'f\Vert_{L^p}= C \Vert u(t,z) \Vert_{L^p_z[\mathbb{S}^{d},\,L^2_t(\mathbb{T})]}\lesssim \Vert u(t,z) \Vert_{L^p_z[\mathbb{S}^{d},\,L^q_t(\mathbb{T})]}.
\end{align*}
\end{proof}
The following lemma, about the $L^2\rightarrow L^p$-operator norm of the spectral projection for the Laplacian on the sphere, is a consequence of the main results in Kwon and Lee \cite{KS}.
\begin{lemma}\label{mainlemmata}
Let $n\in \mathbb{N},$ and let $2\leq p\leq \infty.$ Let us consider the orthogonal projection $H_{n}^d:L^2(\mathbb{S}^d)\rightarrow \mathcal{H}^d_n,$ $n\in \mathbb{N}.$ Then we have
\begin{equation}
    \Vert H_{n}^df \Vert_{L^p(\mathbb{S}^{d})}\leq Cn^{\varkappa_p}\Vert f\Vert_{L^2(\mathbb{S}^{d})},
\end{equation}where $$\varkappa_p:=\frac{d-1}{2}\left(\frac{1}{2}-\frac{1}{p}\right),\quad \textnormal{ if, }2\leq p\leq \frac{2(d+1)}{d-1},$$ and  $$ \varkappa_p: =d\left(\frac{1}{2}-\frac{1}{p}\right)-\frac{1}{2},\quad \textnormal{ if, } p>\frac{2(d+1)}{d-1}. $$  The exponent $\varkappa_p$ is sharp, in the sense that there is not $f\in L^2(\mathbb{S}^d),$ $f\neq 0,$ such that 
\begin{equation}
    \Vert H_{n}^df \Vert_{L^p(\mathbb{S}^{d})}\leq C'n^{s}\Vert f\Vert_{L^2(\mathbb{S}^{d})},
\end{equation}
for all $s<\varkappa_p.$
\end{lemma}
Now, with the analysis developed above and by using Theorem \ref{mainlemmata} we will provide a short proof for the  main result of this section.
\begin{proof}[Proof of Theorem \ref{duvantheorem}]
We observe that from Lemma \ref{T1}, $$ C'_{p,s}\Vert f\Vert_{{F}^0_{p,2}}\leq \Vert u(t,z)\Vert_{L^p_{z}(\mathbb{S}^{d}, L^q_{t}(\mathbb{T}))} \leq C_{p,s}\Vert f\Vert_{{F}^s_{p,2}},  $$ so, we only need to estimate the ${F}^s_{p,2}$-norm of the initial data $f,$ by showing that $\Vert f\Vert_{{F}^s_{p,2}}\lesssim \Vert f\Vert_{W^{s}},$ for every $s\geq\varkappa_{p,q}. $ Moreover, by the embedding $W^{s}\hookrightarrow W^{\varkappa_{p,q}}$ for every $s\geq \varkappa_{p,q}$ we only need to check the previous estimate when $s=\varkappa_{p,q}.$ By the condition $2\leq p<\infty,$  together with  the Minkowski integral inequality and Theorem \ref{mainlemmata}, we  give
\begin{align*}
\Vert f\Vert_{{F}^{s_q}_{p,2}} &=\left\Vert \left( \sum_{n=0}^{\infty} {n}^{2s_q}|H_{n}^df(z)|^2  \right)^{\frac{1}{2}}\right\Vert_{L^p}\\&\leq \left( \sum_{n=0}^{\infty} {n}^{2s_q}\Vert H_{n}^df\Vert^2_{L^p}  \right)^{\frac{1}{2}}\\
&\lesssim \left( \sum_{n=0}^{\infty} {n}^{2(s_q+\varkappa_p)}\Vert H_{n}^df\Vert^2_{L^2}  \right)^{\frac{1}{2}}\\
&=\Vert f\Vert_{W^{\varkappa_{p,q}}},
\end{align*}
where we have used that $\varkappa_{p,q}=s_q+\varkappa_p.$
Let us note that the previous estimates are valid for $p=\infty.$ The sharpness of the result for $q=2$ can be proved by choosing $f=H_n^{d}g$ for some arbitrary, but non-trivial  function $g\in L^2(\mathbb{S}^d).$ In this case, $\varkappa_{p,q}=\varkappa_p$ because $s_q=0,$ and we have 
\begin{align*}
    \Vert u(t,z) \Vert_{L^{p}_z[\mathbb{S}^{d},\,L^2_t(\mathbb{T})]}&=\Vert H_n^{d}g\Vert_{L^p}
    \leq n^{\varkappa_p}\Vert H_n^{d} g\Vert_{L^2}\\
    &=\Vert g\Vert_{W^{\varkappa_p}}.
\end{align*}The sharpness of this inequality is consequence of the sharpness for the exponent $\varkappa_p$ in Lemma \ref{mainlemmata}. Indeed,  an improvement of the previous estimate
would yield improved estimates for the spectral projection operator $H_n^{d}$, which is not
possible. Thus, we finish the proof.
\end{proof}

\section{The Schr\"odinger equation with potential.  }\label{potential}

In this section we study the problem,
\begin{equation}\label{PVI'}
   \textnormal{(PVI'):}
     \begin{cases}
      iu_{t}=(-\Delta_{\mathbb{S}^d}+V(x,t))u,&\quad (t,x)\in \mathbb{T}\times \mathbb{S}^d,
      \\
       u(0,x)=f(x), &\quad x\in \mathbb{S}^d
     \end{cases}
\end{equation}for a suitable potential $V.$ 
Our starting point is the following lemma.
\begin{lemma}
The Schr\"odinger propagator satisfies
\begin{equation} \label{aaaa1}
\left\Vert  \int\limits_0 ^t e^{i(t-\tau)\Delta_{\mathbb{S}^d}} \left( G(\cdot,\tau) \right)(x)d\tau\right\Vert_{L^p_x(\mathbb{S}^d,L^2_t(\mathbb{T}))}\leq  C  \Vert G \Vert_{L^{p'}_x(\mathbb{S}^d,L^2_t(\mathbb{T}))}
\end{equation}
\end{lemma}
\begin{proof}
Using the duality argument, we have
\begin{equation}
\begin{array}{l}
\displaystyle \left\Vert  \int\limits_0 ^t e^{i(t-\tau)\Delta_{\mathbb{S}^d}} \left( G(\cdot,\tau)\right)(x)d\tau\right\Vert_{L^p_x(\mathbb{S}^d,L^2_t(\mathbb{T}))}\\
\displaystyle =\sup \left\{ \left|\,  \int \limits_{\mathbb{S}^d} \int\limits_0^{2\pi}  \int\limits_0 ^t e^{i(t-\tau)\Delta_{\mathbb{S}^d}} \left( G(\cdot,\tau)\right)(x) d\tau h(x,t)dt \ dx\right|:\Vert h \Vert_{L^{p'}_x(\mathbb{S}^d,L^2_t(\mathbb{T}))}=1\right\}.
\end{array} \label{duality}
\end{equation}
Now, since $\Delta_{\mathbb{S}^d}$ is a positive operator, this implies $e^{it\Delta_{\mathbb{S}^d}}$ is unitary, this by various applications of Fubini's theorem  we have
$$
\begin{array}{l}
\displaystyle \int\limits_{\mathbb{S}^d} \int\limits_0^{2\pi}  \int\limits_0 ^t e^{i(t-\tau)\Delta_{\mathbb{S}^d}} \left( G(\cdot,\tau)\right)(x) d\tau h(x,t)dt \ dx\\
 \hspace{4cm}= \displaystyle \int\limits_0^{2\pi}  \int\limits_0 ^t  \int _{\mathbb{S}^d} e^{i(t-\tau)\Delta_{\mathbb{S}^d}} \left( G(x,\tau)\right)(x)h(x,t)dx\ d\tau \ dt\\
  \hspace{4cm}= \displaystyle \int\limits_{\mathbb{S}^d}  \left(\int\limits_0 ^t  e^{-i\tau\Delta_{\mathbb{S}^d} } G(x,\tau) d\tau\right) \left(  \int\limits_0^{2\pi}e^{it \Delta_{\mathbb{S}^d}} h(x,t)dt\right) dx,
\end{array}
$$
via Hölder's inequality we have  
\begin{equation}
\begin{array}{l}
\displaystyle \left|\,  \int \limits_{\mathbb{S}^d} \int\limits_0^{2\pi}  \int\limits_0 ^t e^{i(t-\tau)\Delta_{\mathbb{S}^d}} \left( G(\cdot,\tau)\right)(x) d\tau h(x,t)dt \ dx\right|\\
\hspace{2cm}\displaystyle \leq \left\Vert  \int\limits_0 ^{2\pi} e^{-i\tau\Delta_{\mathbb{S}^d} } G(x,\tau) d\tau \right\Vert _{ L^2(\mathbb{S}^d)}\left\Vert  \int\limits_0 ^{2\pi} e^{it\Delta_{\mathbb{S}^d} } h(x,t) dt \right\Vert  _{ L^2(\mathbb{S}^d)}.
\end{array}\label{aaa}
\end{equation}
Since
$$\left\Vert  \int\limits_0 ^{2\pi} e^{-i\tau\Delta_{\mathbb{S}^d} } G(x,t) dt \right\Vert  _{ L^2(\mathbb{S}^d)} \leq C\Vert G \Vert_{L^{p'}_x(\mathbb{S}^d,L^2_t(\mathbb{T}))}$$
from \eqref{duality} and \eqref{aaa} we have \eqref{aaaa1}
\end{proof}
We define the Banach space
\begin{equation}
X:={C}(\mathbb{T}, W_x^{s}(\mathbb{S}^d))\cap L^p_x(\mathbb{S}^d, L^2_t(\mathbb{T}))
\end{equation}
via the norm 
\begin{equation}
\Vert u\Vert_{X}:=\sup _{t\in \mathbb{T}}\Vert u(t,\cdot)\Vert_{L^2_z}+ \Vert u \Vert_{L^p_x(\mathbb{S}^d,L^2_t(\mathbb{T}))}.
\end{equation}

Now we prove the main result of this section.

\begin{proof}[Proof of Theorem \ref{LiliTheorem}] In view of the Duhamel principle, the solution for (PVI') can be written as
\begin{equation}
    u(t,z)=e^{it \Delta_{\mathbb{S}^d}  }f(x)-\int\limits_{0}^te^{i(t-\tau) \Delta_{\mathbb{S}^d}}(V(\cdot,\tau)u(\cdot))(x)d\tau. \
\end{equation}
We use the standard contraction mapping argument, we need to prove  the nonlinear map 
\begin{equation}
\Phi w:=e^{it \Delta_{\mathbb{S}^d}  }f(x)-\int\limits_{0}^te^{i(t-\tau) \Delta_{\mathbb{S}^d}}(V(\cdot,\tau)w(\cdot))(x)d\tau.\label{operator}
\end{equation}
defines a contraction map on $X$. 
From the unitary group properties and the estimation given in Theorem \ref{duvantheorem} there exists $C_0>1$ such that
\begin{equation}\label{eee1}
\Vert e^{it \Delta_{\mathbb{S}^d}  }f(x) \Vert_{X}\leq C_0 \Vert f\Vert_{W^{s}(\mathbb{S}^d)}.
\end{equation}
Now we  estimate the integral part in \eqref{operator}. Using \eqref{aaaa1} and the estimation in Theorem \ref{duvantheorem} for $q=2$, we also have 
\begin{equation*}
\left\Vert  \int\limits_0 ^t e^{i(t-\tau)\Delta_{\mathbb{S}^d}} \left( V(\tau,\dot)w(\cdot) \right)(x)d\tau\right\Vert_{X}\leq  (C_0+C_0^2)  \Vert  Vw \Vert_{X},
\end{equation*}
Note, $\Vert  Vw \Vert_{L^{p'}_x(\mathbb{S}^d,L^2_t(\mathbb{T}))},\leq \Vert V \Vert_{L^{q}_x(\mathbb{S}^d,L^\infty_t(\mathbb{T}))}\Vert  w \Vert_{L^{p}_x(\mathbb{S}^d,L^2_t(\mathbb{T}))},$ as long as $\frac{1}{q}+\frac{2}{p}=1$, therefore
\begin{equation}
\left\Vert  \int\limits_0 ^t e^{i(t-\tau)\Delta_{\mathbb{S}^d}} \left( V(\tau,\dot)w(\cdot) \right)(x)d\tau\right\Vert_{X}\leq  (C_0+C_0^2)  \Vert V \Vert_{L^{q}_x(\mathbb{S}^d,L^\infty_t(\mathbb{T}))}\Vert  w \Vert_{X}.\label{eee2}
\end{equation}
Hence 
\begin{equation}
    \Vert \Phi w \Vert\leq C_0 \Vert f\Vert_{W^{s}(\mathbb{S}^d)}+ (C_0+C_0^2)  \Vert V \Vert_{L^{q}_x(\mathbb{S}^d,L^\infty_t(\mathbb{T}))}\Vert  w \Vert_{X}.
\end{equation}
So, if we take  $\Vert V \Vert_{L^{q}_x(\mathbb{S}^d,L^\infty_t(\mathbb{T}))}$, being small enough, such that 
$$(C_0+C_0^2)  \Vert V \Vert_{L^{q}_x(\mathbb{S}^d,L^\infty_t(\mathbb{T}))} \leq \frac{1}{2},$$
we observe, $\Phi$ maps $B=\{w\in X: \Vert w\Vert_X\leq 2(C_0+1)\Vert f\Vert_{W^s(\mathbb{S}^d)}\}$ into itself. Now, if $w, v\in B$, the same argument shows that
\begin{equation}
    \Vert \Phi(w)-\Phi(v) \Vert_{X}\leq \frac{1}{2}\Vert w-v\Vert_{X}
\end{equation}
Thus, by the contraction mapping principle, there exists a unique solution to \eqref{PVI'}. Thus, we finish the proof.

\end{proof}

\bibliographystyle{amsplain}

\end{document}